\documentclass[reqno,english]{amsart}
\usepackage{amsfonts,amsmath,latexsym,verbatim,amscd,mathrsfs,color,array}

\usepackage{amsmath,amssymb,amsthm,amsfonts,graphicx,color}
\usepackage{amssymb}
\usepackage{pdfsync}
\usepackage{epstopdf}
\usepackage[colorlinks=true]{hyperref}
\usepackage{subfigure}

\makeatletter
\def\@currentlabel{2.1}\label{e:dispaa}
\def\@currentlabel{2.21}\label{e:dispau}
\def\@currentlabel{2.22}\label{e:dispav}
\def\@currentlabel{2.23}\label{e:dispaw}
\def\@currentlabel{2.24}\label{e:dispax}
\def\theequation{\thesection.\@arabic\c@equation}
\makeatother
\let\oldbibliography\thebibliography
\renewcommand{\thebibliography}[1]{%
\oldbibliography{#1}%
\setlength{\itemsep}{0pt}%
}

\renewcommand{\theequation}{\thesection.\arabic{equation}}
\newtheorem{lemma}{Lemma}[section]

\newtheorem{definition}{Definition}

\newtheorem{proposition}{Proposition}[section]
\newtheorem{corollary}{Corollary}[section]
\newtheorem{remark}{Remark}[section]

\newtheorem{open problem}{Open Problem}[section]
\newtheorem{open question}{Open Quesion}[section]
\newcommand{\bremark}{\begin{remark} \em}
\newcommand{\eremark}{\end{remark} }

\newtheorem{numerical/experimental results}{Numerical/Experimental results}[section]
\newtheorem*{thmA}{Theorem A}
\newtheorem*{thmB}{Theorem B}

\newtheorem{theorem}{Theorem}[section]

\newcommand{\BE}{\begin{equation}}
\newcommand{\BEN}{\begin{equation*}}
\newcommand{\EE}{\end{equation}}
\newcommand{\EEN}{\end{equation*}}
\newcommand{\BL}{\begin{lemma}}
\newcommand{\EL}{\end{lemma}}
\newcommand{\BT}{\begin{theorem}}
\newcommand{\ET}{\end{theorem}}
\newcommand{\BP}{\begin{proposition}}
\newcommand{\EP}{\end{proposition}}
\newcommand{\BC}{\begin{corollary}}
\newcommand{\EC}{\end{corollary}}

\renewcommand{\Im}{\operatorname{Im}}

\usepackage{amsmath}

\DeclareMathOperator*{\argmin}{arg\,min}

\begin{document}


\title[Signs of high order derivatives]{\bf signs of high order derivatives for the theta and Epstein zeta functions and application}

\author{Kaixin Deng}

\author{Senping Luo}

\address[K.~Deng]{School of Mathematics and statistics, Jiangxi Normal University, Nanchang, 330022, China}
\address[S.~Luo]{School of Mathematics and statistics, Jiangxi Normal University, Nanchang, 330022, China}

\email[S.~Luo]{luosp1989@163.com}

\email[K.~Deng]{Dengkaikai1999@126.com}

\begin{abstract}
In the 1950s, 1960s and 1988, number theorists Rankin \cite{Ran1953}, Cassels \cite{Cas1959}, Ennola \cite{Enn1964a}, Diananda \cite{Dia1964},
and Montgomery \cite{Mon1988} derived the signs of first order derivatives of Epstein zeta and theta functions, respectively. In this note, we shall derive the signs of higher order derivatives of such functions. Application to lattice minimization problems will be given.
\end{abstract}

\maketitle


\section{Introduction and Statement of Main Results}

 The Epstein zeta and theta functions associated with the lattice $ \Lambda$ are defined as follows:
 \begin{equation}\aligned
 \nonumber
\zeta(s;\Lambda):=\sum_{\mathbb{P}\in \Lambda\setminus\{0\}}\frac{1}{|\mathbb{P}|^{2s}}, \; s>1;\;\;\theta(\alpha;\Lambda):=\sum_{\mathbb{P}\in\Lambda} e^{-\pi\alpha |\mathbb{P}|^2},\; \alpha>0.
\endaligned\end{equation}
These functions coincide with the $f-$potential lattice energy
 \begin{equation}\aligned\label{EFL}
E_f(\Lambda):=\sum_{\mathbb{P}\in \Lambda \backslash\{0\}} f(|\mathbb{P}|^2),\;\; |\cdot|\;\hbox{is the Euclidean norm on}\;\mathbb{R}^2
\endaligned\end{equation}
for
  \begin{equation}\aligned\label{EFL000}
 f(r^2)=\frac{1}{r^{2s}},\;s>1 \;\;\hbox{and}\;\;f(r^2)=e^{- \pi\alpha r^2},\;\alpha>0,
 \endaligned\end{equation}
which represent the Riesz and Gaussian potentials, respectively (see \cite{Bet2016}).

Let $ z\in \mathbb{H}:=\{z= x+ i y\in\mathbb{C}: y>0\}$ and  $\Lambda :=\sqrt{\frac{1}{\Im(z)}}\Big({\mathbb Z}\oplus z{\mathbb Z}\Big)$  be a lattice in $ \mathbb{R}^2$  with unit cell area and parameterized by $z$.  By the parametrization $\Lambda =\sqrt{\frac{1}{\Im(z)}}\Big({\mathbb Z}\oplus z{\mathbb Z}\Big)$, we have

 \begin{equation}\aligned
 \label{thetas}
\zeta(s;z):=&\zeta(s;\Lambda)=\sum_{(m,n)\in\mathbb{Z}^2\backslash\{0\}}\frac{y^s}{|mz+n|^{2s}},\; s>1;\\
\theta(\alpha;z):=&\theta(\alpha;\Lambda)=\sum_{(m,n)\in\mathbb{Z}^2}e^{-\pi \alpha\frac{|mz+n|^{2}}{y}},\; \alpha>0.
\endaligned\end{equation}

The functionals $\theta(\alpha;z)$ and $\zeta(s;z)$ exhibit invariance under the action of a certain group. The generators of the group are given by
\begin{equation}\aligned\label{GroupG1}
\mathcal{G}: \hbox{the group generated by} \;\;\tau\mapsto -\frac{1}{\tau},\;\; \tau\mapsto \tau+1,\;\;\tau\mapsto -\overline{\tau}.
\endaligned\end{equation}
Therefore,
\begin{equation}\aligned\label{Gabc}
\theta(\alpha;\gamma(z))=\theta(\alpha;z),\;\;
\zeta(\alpha;\gamma(z))=\zeta(\alpha;z)\;\;\hbox{for any}\;\;\gamma\in\mathcal{G},\;\hbox{provided}\;\alpha>0\;\hbox{and}\;s>1.
\endaligned\end{equation}
Within the invariance in \eqref{Gabc}, one can reduce the consideration of domain from the upper half plane $\mathbb{H}$ to their fundamental domain.
\begin{definition}[\cite{Apo1976,Eva1973}]\label{Def}
The fundamental domain associated to group $G$ is a connected domain $\mathcal{D}$ that satisfies
\begin{itemize}
  \item For any $z\in\mathbb{H}$, there exists an element $\pi\in G$ such that $\pi(z)\in\overline{\mathcal{D}}$;
  \item Suppose $z_1,z_2\in\mathcal{D}$ and $\pi(z_1)=z_2$ for some $\pi\in G$, then $z_1=z_2$ and $\pi=\pm Id$.
\end{itemize}
\end{definition}

By Definition \ref{Def}, the fundamental domain associated to modular group $\mathcal{G}$ is
\begin{equation}\aligned\label{D}
\mathcal{D}_{\mathcal{G}}:=\{
z\in\mathbb{H}: |z|>1,\; 0<x<\frac{1}{2}
\}.
\endaligned\end{equation}

Recent years, it has been revealed that the functionals $\theta(\alpha;z)$ and $\zeta(s;z)$ are deeply connected to a long-standing open crystal problem: understanding the fundamental mechanisms behind the spontaneous arrangement of atoms into periodic configurations at low temperatures (Radin \cite{Radin1987}). This famous problem, called as the Crystallization Conjecture, was proposed by Radin in 1987, a comprehensive review of this conjecture can be found in Blanc-Lewin \cite{Blanc2015}. For applications of $\theta(\alpha;z)$ and $\zeta(s,z)$ to the crystal problems,
see B\'etermin and his collaborators \cite{Bet2015}-\cite{Betermin2021LMP}, as well as Wei and his collaborators \cite{Luo2019}-\cite{Luo_arxiv_2021}.

In the 1950s and 1960s, number theorists Rankin \cite{Ran1953}, Cassels \cite{Cas1959}, Ennola \cite{Enn1964a,Enn1964b}, and Diananda \cite{Dia1964} established the following result:
\begin{thmA}[Rankin, Cassels, Ennola, Diananda, 1950s and 1960s] For $s>1$, up to the action by the modular group,
$$
\argmin_{z\in\mathbb{H}}\zeta(s;z)=e^{i\frac{\pi}{3}}.
$$
\end{thmA}

Motivated by {\bf Theorem A}, in 1988,
Montgomery \cite{Mon1988} further proved:
\begin{thmB}[Montgomery] For $\alpha>0$, up to the action by the modular group,
$$
\argmin_{z\in\mathbb{H}}\theta(\alpha;z)=e^{i\frac{\pi}{3}}.
$$
\end{thmB}

The proofs of Theorems A and B are based on determination of the signs of the first order derivatives of $\zeta(s,z)$ and $\theta(\alpha,z)$ as follows:
\begin{proposition}[\cite{Cas1959,Dia1964,Enn1964a,Enn1964b,Mon1988,Ran1953}] Signs of first order derivatives of $\zeta(s,z)$ and $\theta(\alpha,z)$ in fundamental domain.
\begin{itemize}
  \item [(1)] $\frac{\partial}{\partial y}\zeta(s,z)>0$ for $y\geq\frac{3}{2}$ and $s>1$.
  \item [(2)] $\frac{\partial}{\partial x}\zeta(s,z)<0$ for $y\geq\frac{3}{5}, 0<x<\frac{1}{2}$ and $s>1$.
  \item [(3)] $\frac{\partial}{\partial y}\theta(s,z)\geq0$ for $z\in\overline{\mathcal{D}_{\mathcal{G}}}$ and $\alpha>0$.
   \item [(4)] $\frac{\partial}{\partial x}\theta(s,z)\leq0$ for $y\geq\frac{1}{2}, 0<x<\frac{1}{2}$ and $\alpha>0$.

\end{itemize}
\end{proposition}
In this note, we aim to determine the signs of higher order derivatives of $\zeta(s;z)$ and $\theta(\alpha;z)$ and provide the consequence in corresponding minimization problems. Our main results are as follows:
\begin{theorem}\label{Th1}
 Assume that $\alpha>0$ and $s>1$. Then
\begin{itemize}
  \item [(1)] For $z\in \{(x,y)|\;0<x<\frac{1}{2},y\geq\frac{3}{5}\}\supseteq D_{\mathcal{G}}$, it holds that
  \begin{equation}\aligned\nonumber
\frac{\partial^2}{\partial x \partial y}\theta(\alpha;z)>0,\;
\frac{\partial^2}{\partial x \partial y}\zeta(s;z)>0.
\endaligned\end{equation}

  \item [(2)] For $z\in D_{\mathcal{G}}$, it holds that

  \begin{equation}\aligned\nonumber
\frac{\partial^3}{\partial x \partial y^2}\theta(\alpha;z)<0,\;
\frac{\partial^3}{\partial x \partial y^2}\zeta(s;z)<0.
\endaligned\end{equation}

\end{itemize}
\end{theorem}

As a direct consequence of Theorem \ref{Th1}, we have
\begin{corollary}\label{Cor1}
 Assume that $\alpha>0$ and $s>1$. Let $\Gamma:=\{z\in\mathbb{H}: z=e^{i\theta},\; \theta\in[\frac{\pi}{3},\frac{\pi}{2}]\}$. Then
  \begin{equation}\aligned\nonumber
&\min_{z\in\overline{D_{\mathcal{G}}}}\frac{\partial}{\partial x}\theta(\alpha;z)=
\min_{z\in\Gamma}\frac{\partial}{\partial x}\theta(\alpha;z),\;
\min_{z\in\overline{D_{\mathcal{G}}}}\frac{\partial}{\partial x}\zeta(s;z)=
\min_{z\in\Gamma}\frac{\partial}{\partial x}\zeta(s;z);\\
&\min_{z\in\overline{D_{\mathcal{G}}}}\frac{\partial^2}{\partial x \partial y}\theta(\alpha;z)
=\min_{z\in\Gamma}\frac{\partial^2}{\partial x \partial y}\theta(\alpha;z)
,\;
\min_{z\in\overline{D_{\mathcal{G}}}}\frac{\partial^2}{\partial x \partial y}\zeta(s;z)
=\min_{z\in\Gamma}\frac{\partial^2}{\partial x \partial y}\zeta(s;z)
.
\endaligned\end{equation}

\end{corollary}

\section{Estimates on 1-d theta functions}

In this section, we collect some known estimates on 1-d theta functions. We also derive some useful estimates here which have their independent interest.

Recall that the classical one-dimensional theta function is given by
\begin{equation}\aligned\label{TXY}
\vartheta(X;Y)=\sum_{n=-\infty}^\infty e^{-\pi n^2 X} e^{2n\pi i Y}.
 \endaligned\end{equation}
By the Poisson summation formula, it holds that
\begin{equation}\aligned\label{Poisson}
\vartheta(X;Y)=X^{-\frac{1}{2}}\sum_{n=-\infty}^\infty e^{-\pi \frac{(n-Y)^2}{X}} .
 \endaligned\end{equation}

To estimate bounds of quotients of derivatives of $\vartheta(X;Y)$, we denote that
\begin{equation}\aligned\label{duct}
\mu(X)&=\sum_{n=2}^{\infty}n^2e^{-\pi (n^2-1) X},\:\:\:\:\:\:\:\:\:\:\:\:\hat{\mu}(X)=\sum_{n=2}^{\infty}n^2e^{-\pi (n^2-1) X}(-1)^{n+1},\\
\nu(X)&=\sum_{n=2}^{\infty}n^4e^{-\pi (n^2-1) X},\:\:\:\:\:\:\:\:\:\:\:\:\hat{\nu}(X)=\sum_{n=2}^{\infty}n^4e^{-\pi (n^2-1) X}(-1)^{n+1},\\
\omega(X)&=\sum_{n=2}^{\infty}n^6e^{-\pi (n^2-1) X},\:\:\:\:\:\:\:\:\:\:\:\hat{\omega}(X)=\sum_{n=2}^{\infty}n^6e^{-\pi (n^2-1) X}(-1)^{n+1}.
 \endaligned\end{equation}.

We introduce some useful estimates for the quotient of the derivative of 1-d theta functions.

\begin{lemma}[\cite{Luo2023,DL2024}]\label{2lem1}
Assume that $Y>0$ and $k\in \mathbb{N^{+}}$. It holds that
\begin{itemize}
  \item [(1)] For $ X>\frac{1}{5}$, then $\left|\frac{\vartheta_{Y}(X;kY)}{\vartheta_{Y}(X;Y)} \right|\leq k\cdot \frac{1+\mu(X)}{1-\mu(X)}.$
  \item [(2)] For $X< \frac{\pi}{\pi+2} $, then $\left|\frac{\vartheta_{Y}(X;kY)}{\vartheta_{Y}(X;Y)} \right|\leq  \frac{k}{\pi}e^{\frac{\pi}{4X}}.$

  \item [(3)] For $X\geq \frac{1}{5}$, then $-{\pi}\cdot \frac{1+{\nu}(X)}{1+{\mu}(X)}\leq \frac{\vartheta_{XY}(X;Y)}{\vartheta_{Y}(X;Y)} \leq -{\pi} \cdot\frac{1+\hat{\nu}(X)}{1+\hat{\mu}(X)}.$
  \item [(4)] For $0<X\leq\frac{1}{2}$, then $\frac{\frac{3}{4}{X}^2+2{\pi}^2 e^{-\frac{\pi}{X}}}{-\frac{1}{2}{X}^3+2\pi{X}^2e^{-\frac{\pi}{X}}}\leq \frac{\vartheta_{XY}(X;Y)}{\vartheta_{Y}(X;Y)} \leq \frac{\pi}{4{X}^2}.$
  \item [(5)] For $ X\geq \frac{1}{5}$, then $\left|\frac{\vartheta_{XY}(X;kY)}{\vartheta_{Y}(X;Y)} \right|\leq k\pi \cdot \frac{1+\nu(X)}{1-\mu(X)}.$
  \item [(6)] For $0<X\leq \frac{1}{2}$, then $\left|\frac{\vartheta_{XY}(X;kY)}{\vartheta_{Y}(X;Y)} \right|\leq \frac{3k}{2\pi} X^{-1}(1+\frac{\pi}{6}\frac{1}{X})e^{\frac{\pi}{4X}}.$
  \item [(7)] For $X\geq \frac{59}{250},$
  ${\pi}^2\cdot \frac{1+\hat{\omega}(X)}{1+\hat{\mu}(X)}\leq \frac{\vartheta_{XXY}(X;Y)}{\vartheta_{Y}(X;Y)} \leq {\pi}^2 \cdot\frac{1+\omega(X)}{1+\mu(X)}.$
 \end{itemize}
 \end{lemma}

Further, we provide two estimates for the quotient of the derivative of 1-d theta functions.
\begin{lemma} \label{2lem2} For $0<X\leq \frac{1}{2},\;Y>0,$ and $k\in \mathbb{N^{+}}$, it holds that
\begin{itemize}
      \item [(1)] $\frac{15}{4X^2}-\frac{\pi}{4X^4}\leq \frac{\vartheta_{XXY}(X;Y)}{\vartheta_{Y}(X;Y)}\leq \frac{15}{4X^2}+\frac{\pi}{4X^4}.$
            \item [(2)] $\left|\frac{\vartheta_{XXY}(X;kY)}{\vartheta_{Y}(X;Y)} \right|\leq \frac{k}{4X^2}(\frac{15}{\pi}+\frac{1}{X^2})e^{\frac{\pi}{4X}}.$
            \end{itemize}
\end{lemma}
\begin{proof}
By \eqref{Poisson}, one gets
\begin{equation}\aligned\label{2lem2eq1}
\frac{\vartheta_{XXY}(X;Y)}{\vartheta_{Y}(X;Y)}
=\frac{15}{4}X^{-2} +\pi X^{-4}\big( \pi\cdot\frac{\underset{n\in \mathbb{Z}}{\sum}(n-Y)^5 e^{-\frac{\pi(n-Y)^2}{X}}}
{\underset{n\in \mathbb{Z}}{\sum}(n-Y)e^{-\frac{\pi(n-Y)^2}{X}}}-5X\cdot \frac{\underset{n\in \mathbb{Z}}{\sum}(n-Y)^3 e^{-\frac{\pi(n-Y)^2}{X}}}
{\underset{n\in \mathbb{Z}}{\sum}(n-Y)e^{-\frac{\pi(n-Y)^2}{X}}} \big).
\endaligned\end{equation}
Combining \eqref{2lem2eq1} with Lemma \ref{2lem4} (in the following), we obtain (1). By the following deformation:
\begin{equation}\aligned\nonumber
\left|\frac{\vartheta_{XXY}(X;kY)}{\vartheta_{Y}(X;Y)}\right|
=\left|\frac{\vartheta_{XXY}(X;kY)}{\vartheta_{Y}(X;kY)}\right|\cdot\left|\frac{\vartheta_{Y}(X;kY)}{\vartheta_{Y}(X;Y)}\right|,
\endaligned\end{equation}
and using (1) along with Lemma \ref{2lem1}, we arrive at (2).
\end{proof}

To prove Lemma \ref{2lem2}, we need the following three auxiliary lemmas.
\begin{lemma}[\cite{Luo2022,Luo2023,DL2024}]\label{2lem3} Assume that $k\in \mathbb{Z}$. Then it holds that
\begin{itemize}
\item [(1)] ${\sum}_{n\in \mathbb{Z}}(n-Y)e^{-\frac{ \pi (n-Y)^2}{X}}\leq 0$ for $X>0$ and $Y\in[k,k+\frac{1}{2}]$.
\item [(2)] ${\sum}_{n\in \mathbb{Z}}(n-Y)e^{-\frac{ \pi (n-Y)^2}{X}}\geq0$ for $X>0$ and $Y\in[k-\frac{1}{2},k]$.
\item [(3)] $\underset{X\in(0,\frac{1}{2}],Y\in \mathbb{R}}{\sup}\left|   \frac{\sum_{n\in\mathbb{Z}}(n-Y)^3 e^{-\frac{\pi(n-Y)^2}{X}}}{\sum_{n\in \mathbb{Z}}(n-Y)e^{-\frac{\pi(n-Y)^2}{X}}}        \right|\leq\frac{1}{4}.$
      \item [(4)] $
\underset{X\in(0,\frac{1}{2}],Y\in \mathbb{R}}{\sup}\left|   \frac{\sum_{n\in\mathbb{Z}}(n-Y)^5 e^{-\frac{\pi(n-Y)^2}{X}}}{\sum_{n\in \mathbb{Z}}(n-Y)e^{-\frac{\pi(n-Y)^2}{X}}}        \right|\leq\frac{1}{16}.$
  \end{itemize}
\end{lemma}
\begin{proof}
Luo-Wei \cite{Luo2022,Luo2023} imply (1) and (2). In fact, it proves that
\begin{equation}\aligned\label{lem2.3eq2}
&\frac{\partial}{\partial Y}\vartheta(X;Y)\leq0,\;\;\mathrm{ for}\;\ X>0\; \mathrm{and}\; Y\in[k,k+\frac{1}{2}],\\
&\frac{\partial}{\partial Y}\vartheta(X;Y)\geq0,\;\;\mathrm{ for}\;\ X>0\; \mathrm{and}\; Y\in[k-\frac{1}{2},k].\\
\endaligned\end{equation}
By \eqref{Poisson}, one has
\begin{equation}\aligned\label{lem2.3eq1}
\frac{\partial}{\partial Y}\vartheta(X;Y)=2\pi X^{-\frac{3}{2}}\underset{n\in \mathbb{Z}}{\sum}(n-Y)e^{-\frac{\pi (n-Y)^2}{X}}.
\endaligned\end{equation}
Therefore, \eqref{lem2.3eq2} and \eqref{lem2.3eq1} yield (1) and (2). \cite{Luo2023} and \cite{DL2024} presented (3) and (4), respectively.
\end{proof}

With Lemma \ref{2lem3}, we give the following estimate directly used in Lemma \ref{2lem2}.

\begin{lemma}\label{2lem4} For $0<X\leq \frac{1}{2}$ and $Y\in \mathbb{R}$, it holds that
\begin{equation}\aligned\nonumber
\left|\pi\cdot\frac{\underset{n\in \mathbb{Z}}{\sum}(n-Y)^5 e^{-\frac{\pi(n-Y)^2}{X}}}
{\underset{n\in \mathbb{Z}}{\sum}(n-Y)e^{-\frac{\pi(n-Y)^2}{X}}}-5X\cdot \frac{\underset{n\in \mathbb{Z}}{\sum}(n-Y)^3 e^{-\frac{\pi(n-Y)^2}{X}}}
{\underset{n\in \mathbb{Z}}{\sum}(n-Y)e^{-\frac{\pi(n-Y)^2}{X}}}\right|\leq \frac{1}{4}.
\endaligned\end{equation}
 \end{lemma}
\begin{proof}
Let $a=\frac{1}{X}$, and $\mathcal{Q}(a;Y):=\pi\cdot\frac{\underset{n\in \mathbb{Z}}{\sum}(n-Y)^5 e^{-a \pi(n-Y)^2}}
{\underset{n\in \mathbb{Z}}{\sum}(n-Y)e^{-a \pi(n-Y)^2}}-\frac{5}{a}\cdot \frac{\underset{n\in \mathbb{Z}}{\sum}(n-Y)^3 e^{-a \pi(n-Y)^2}}
{\underset{n\in \mathbb{Z}}{\sum}(n-Y)e^{-a \pi(n-Y)^2}},$ then $a\geq2.$ By a direct check, we find that $\mathcal{Q}(a;Y)$ satisfies the following properties:
\begin{equation}\aligned\nonumber
\mathcal{Q}(a;Y)=\mathcal{Q}(a;Y+1),\;\mathcal{Q}(a;Y)=\mathcal{Q}(a;1-Y).
\endaligned\end{equation}
Therefore, this lemma is reduced to proving that
\begin{equation}\aligned\label{lem2.3equiv1}
\underset{a\geq2,Y\in[0,\frac{1}{2}]}{\sup} \left|\mathcal{Q}(a;Y) \right|\leq \frac{1}{4}.
\endaligned\end{equation}
By Lemma \ref{2lem3}, one has
\begin{equation}\aligned\label{lem2.3eq3}
\left|\mathcal{Q}(a;Y) \right|\leq \frac{\pi}{16}+\frac{5}{4a} \leq \frac{1}{4}, \;\;\mathrm{if} \;\;a\geq\frac{20}{4-\pi}.
\endaligned\end{equation}
Given \eqref{lem2.3equiv1} and \eqref{lem2.3eq3}, then it remains to consider the case where $a\in[2,\frac{20}{4-\pi}]$ and $Y\in[0,\frac{1}{2}]$. By Lemma \ref{2lem3}, we can transform \eqref{lem2.3equiv1} into
\begin{equation}\aligned\label{lem2.3eq4}
\mathcal{F}(a;Y)\geq0,\;\;\mathcal{H}(a;Y)\leq0,
\endaligned\end{equation}
where
\begin{equation}\aligned\label{lem2.3eq5}
\mathcal{F}(a;Y):=&\underset{n\in \mathbb{Z}}{\sum}\big( \pi (n-Y)^5-\frac{5}{a}(n-Y)^3-\frac{1}{4}(n-Y) \big)e^{-a \pi (n-Y)^2},\\
\mathcal{H}(a;Y):=&\underset{n\in \mathbb{Z}}{\sum}\big( \pi (n-Y)^5-\frac{5}{a}(n-Y)^3+\frac{1}{4}(n-Y) \big)e^{-a \pi (n-Y)^2}.
\endaligned\end{equation}
For $a\in[2,\frac{20}{4-\pi}]$ and $Y\in 0,\frac{1}{2}]$, the approach to proving $\mathcal{H}(a;Y)\leq0$ is analogous to that of $\mathcal{F}(a;Y)\geq0.$ To avoid repetition, we only provide the proof of $\mathcal{F}(a;Y)\geq0$, which is detailed in Lemma \ref{2lem5}.
\end{proof}

\begin{lemma}\label{2lem5} Assume $a\in[2,24]$. It holds that $\mathcal{F}(a;Y)\geq0$ for $Y\in[0,\frac{1}{2}]$. It is split into four items:
\begin{itemize}
\item [(1)] $\mathcal{F}(a;0)=\mathcal{F}(a;\frac{1}{2})=0.$
      \item [(2)] $\frac{\partial}{\partial{Y}}\mathcal{F}(a;Y)\geq \frac{1}{10}e^{-a\pi Y^2}>0$ for $Y\in[0,\frac{1}{20}].$
            \item [(3)] $\frac{\partial}{\partial{Y}}\mathcal{F}(a;Y)\leq-\frac{3}{5}e^{-a\pi Y^2}<0$ for $Y\in[\frac{2}{5},\frac{1}{2}].$
            \item [(4)] $\mathcal{F}(a;Y)\geq \frac{1}{100} e^{-a\pi Y^2} >0$ for $Y\in[\frac{1}{20},\frac{2}{5}].$
                  \end{itemize}
       \end{lemma}
\begin{proof}
Given that $\mathcal{F}(a;Y)$ is symmetric about $Y=0$ and $Y=\frac{1}{2}$ , and the summation expression for $\mathcal{F}(a;Y)$ in \eqref{lem2.3eq5} includes all integer of $n$, we consequently obtain (1).
For convenience, we denote that
\begin{equation}\aligned\label{lem2.5eq0}
f_{n}(a;Y)=&\big( \pi (n-Y)^5-\frac{5}{a}(n-Y)^3-\frac{1}{4}(n-Y) \big)e^{-a \pi (n-Y)^2},\\
f_{n}'(a;Y)=&\big(2 a \pi^2(n-Y)^6-15\pi (n-Y)^4+(\frac{15}{a}-\frac{a\pi}{2})(n-Y)^2+\frac{1}{4}\big)e^{-a\pi (n-Y)^2}.
\endaligned\end{equation}
By \eqref{lem2.3eq5}, one has
\begin{equation}\aligned\label{lem2.5eq1}
\mathcal{F}(a;Y)=\underset{n\in \mathbb{Z}}{\sum}f_{n}(a;Y),\;\;\frac{\partial}{\partial{Y}}\mathcal{F}(a;Y)=\underset{n\in \mathbb{Z}}{\sum}f_{n}'(a;Y).
\endaligned\end{equation}
For $a\in[2,24]$ and $Y\in[0,\frac{1}{2}]$, we observe that:
\begin{equation}\aligned\label{lem2.5eq2}
f_{-n}(a;Y)\leq0,\;f_{n}(a;Y)\geq0\;\;\mathrm{for}\;\;n\geq2;\;\mathrm{and}\;f_{n}'(a;Y)\geq0\;\;\mathrm{for}\;\;|n|\geq2.
\endaligned\end{equation}
For $a\in[2,24]$ and $Y\in[0,\frac{1}{20}]$, by \eqref{lem2.5eq1} and \eqref{lem2.5eq2}, it follows that
\begin{equation}\aligned\label{lem2.5eq3}
\frac{\partial}{\partial{Y}}\mathcal{F}(a;Y)\geq \underset{|n|\leq 1}{\sum}f_{n}'(a;Y)\geq \frac{1}{10}e^{-a\pi Y^2}>0,
\endaligned\end{equation}
which yields (2). Subsequently, for $a\in[2,24]$ and $Y\in[\frac{2}{5},\frac{1}{2}]$, by \eqref{lem2.5eq0} and \eqref{lem2.5eq1}, one has
\begin{equation}\aligned\label{lem2.5eq4}
-\underset{|n|\leq 2}{\sum}f_{n}'(a;Y)=&\underset{|n|\leq 2}{\sum}\big((\frac{a\pi}{2}-\frac{15}{a})(n-Y)^2+15\pi(n-Y)^4-2a\pi(n-Y)^6-\frac{1}{4} \big)e^{-a\pi(n-Y)^2}\\
\geq& \frac{7}{10}e^{-a\pi Y^2}>0.
\endaligned\end{equation}
Furthermore, by a straightforward calculation, we have:
\begin{equation}\aligned\label{lem2.5eq4add}
f_{1-n}'(a;Y)\geq f_{n}'(a;Y)\geq0, \;\; \mathrm{for} \;\;a\in[2,24],\; Y\in[\frac{2}{5},\frac{1}{2}] \;\;\mathrm{and}\;\;n\geq2.
\endaligned\end{equation}
By \eqref{lem2.5eq4add}, one has
\begin{equation}\aligned\label{lem2.5eq5}
\underset{|n|\geq3}{\sum}f_{n}'(a;Y)=\underset{n\geq3}{\sum}f_{n}'(a;Y)+\underset{n\leq -3}{\sum}f_{n}'(a;Y)\leq 2\underset{n\leq -2}{\sum}f_{n}'(a;Y).
\endaligned\end{equation}
By \eqref{lem2.5eq0}, one gets
\begin{equation}\aligned\label{lem2.5eq5add}
\underset{n\leq -2}{\sum}f_{n}'(a;Y)=&\underset{n\geq2}{\sum}\big( 2a \pi^2(n+Y)^6 -15\pi(n+Y)^4+(\frac{15}{a}-\frac{a\pi}{2})(n+Y)^2+\frac{1}{4}\big)e^{-a\pi(n+Y)^2}\\
\leq& \underset{n\geq2}{\sum}\big( 2a \pi^2(n+Y)^6 +\frac{15}{a}+\frac{1}{4}\big)e^{-a\pi(n+Y)^2}
\leq \frac{5}{2}a \pi^2\underset{n\geq2}{\sum} (n+Y)^6 e^{-a\pi(n+Y)^2}.
\endaligned\end{equation}
Then by \eqref{lem2.5eq4},\eqref{lem2.5eq5} and \eqref{lem2.5eq5add}, one gets
\begin{equation}\aligned\label{lem2.5eq5addadd}
\frac{{\sum}_{|n|\geq3}|f_{n}'(a;Y)|}{{\sum}_{|n|\leq2}(-f_{n}'(a;Y))}\leq& \frac{50}{7}a\pi^2e^{a\pi Y^2}\underset{n\geq2}{\sum} (n+Y)^6 e^{-a\pi(n+Y)^2}\\
\leq &\frac{50}{7}a\pi^2(2+Y)^6 e^{-4 a\pi(1+Y)}\big( 1+\underset{n\geq3}{\sum} (\frac{n+Y}{2+Y})^6 e^{-a\pi\big((n+Y)^2-(2+Y)^2\big)}\big)\\
\leq& 8a\pi^2(2+Y)^6 e^{-4 a\pi(1+Y)}\leq 10^{-10}.
\endaligned\end{equation}
Therefore, by \eqref{lem2.5eq1}, \eqref{lem2.5eq4} and \eqref{lem2.5eq5addadd}, we have
\begin{equation}\aligned\label{lem2.5eq5aaa}
-\frac{\partial}{\partial{Y}}\mathcal{F}(a;Y)
\geq\big(-\underset{|n|\leq 2}{\sum}f_{n}'(a;Y) \big)\big( 1-\frac{{\sum}_{|n|\geq3}|f_{n}'(\alpha;Y)|}{{\sum}_{|n|\leq2}(-f_{n}'(\alpha;Y))}\big)\geq\frac{3}{5}e^{-a\pi Y^2}>0,
\endaligned\end{equation}
which deduces (3). Next, we consider (4). For $a\in[2,24]$ and $Y\in[\frac{1}{20},\frac{2}{5}]$, by \eqref{lem2.5eq1} and \eqref{lem2.5eq2},
we have
\begin{equation}\aligned\label{lem2.5eq6}
\underset{|n|\leq1}{\sum}f_{n}(a;Y)\geq\frac{1}{80}e^{-a\pi Y^2}>0,
\endaligned\end{equation}
and
\begin{equation}\aligned\label{lem2.5eq7}
\underset{n\leq-2}{\sum}|f_{n}(a;Y)|=&\underset{n\leq-2}{\sum}\big(-f_{n}(a;Y)\big)=\underset{n\leq-2}{\sum}\big( \pi (Y-n)^5-\frac{5}{a}(Y-n)^3-\frac{1}{4}(Y-n) \big)e^{-a \pi (Y-n)^2}\\
=&\underset{n\geq2}{\sum}\big(\pi(n+Y)^5-\frac{5}{a}(n+Y)^3-\frac{1}{4}(n+Y)\big)e^{-a\pi(n+Y)^2}\leq \underset{n\geq2}{\sum}\pi(n+Y)^5e^{-a\pi(n+Y)^2}.
\endaligned\end{equation}
Then by \eqref{lem2.5eq6} and \eqref{lem2.5eq7}, for $a\in[2,24]$ and $Y\in[\frac{1}{20},\frac{2}{5}]$, we have
\begin{equation}\aligned\label{lem2.5eq8}
\frac{{\sum}_{n\leq-2}|f_{n}(a;Y)|}{{\sum}_{|n|\leq1}f_{n}(a;Y)}\leq &{80}\pi e^{a\pi Y^2}\underset{n\geq2}{\sum}(n+Y)^5 e^{-a\pi(n+Y)^2}\\
\leq &{80}\pi (2+Y)^5e^{-4a\pi(1+Y)}\big(1+\underset{n\geq3}{\sum} (\frac{n+Y}{2+Y})^5 e^{-a\pi\big((n+Y)^2-(2+Y)^2\big)}  \big)\leq 10^{-7}.
\endaligned\end{equation}
Consequently, by \eqref{lem2.5eq2}, \eqref{lem2.5eq6} and \eqref{lem2.5eq8}, we have
\begin{equation}\aligned\nonumber
\mathcal{F}(a;Y)\geq\underset{n\leq-1}{\sum}f_{n}(a;Y)\geq\big(\underset{|n|\leq1}{\sum}f_{n}(a;Y)\big)
\big(1-\frac{{\sum}_{n\leq-2}|f_{n}(a;Y)|}{{\sum}_{|n|\leq1}f_{n}(a;Y)} \big)\geq \frac{1}{100}e^{-a\pi Y^2}>0.
\endaligned\end{equation}
This completes the proof.
\end{proof}

\section{Proof of Theorem \ref{Th1}}

 Prior to proving Theorem \ref{Th1}, we first introduce the exponential expansion of theta function deduced in \cite{Luo2022,Mon1988}.

\begin{lemma}[\cite{Luo2022,Mon1988}]\label{3.1lem1} For $\alpha>0$ and $y>0$, we have the exponential expansion of $\theta (\alpha;z)$:
\begin{equation}\aligned\nonumber
\theta (\alpha;z)
&=\sqrt{\frac{y}{\alpha}}\sum_{n\in\mathbb{Z}}e^{-\alpha \pi y n^2}\vartheta(\frac{y}{\alpha};nx)\\
&=2\sqrt{\frac{y}{\alpha}}\sum_{n=1}^\infty e^{-\alpha \pi y n^2}\vartheta(\frac{y}{\alpha};nx)+\sqrt{\frac{y}{\alpha}}\vartheta(\frac{y}{\alpha};0).
\endaligned\end{equation}
\end{lemma}

By Lemma \ref{3.1lem1}, we obtain the expression for the mixed second partial derivative of theta function.
\begin{lemma} \label{3.1lem2} For $\alpha>0$ and $y>0$, we have the following expression:
\begin{equation}\aligned\nonumber
\frac{\partial^2}{\partial x\partial y}\theta(\alpha;z)=\sum_{n=1}^\infty \Big( \big(\alpha^{-\frac{1}{2}}y^{-\frac{1}{2}} n-2\pi {\alpha}^{\frac{1}{2}}{y}^{\frac{1}{2}}n^3\big) \vartheta_{Y}(\frac{y}{\alpha};nx)
+2{\alpha}^{-\frac{3}{2}} y^{\frac{1}{2}} n \vartheta_{XY}(\frac{y}{\alpha};nx)\Big)e^{-\alpha \pi y n^2}.
\endaligned\end{equation}
\end{lemma}

Using Lemma \ref{3.1lem2}, we further obtain the high-order derivative expression for theta function as follows:

\begin{lemma} \label{3.2lem1} For $\alpha>0$ and $y>0$, we have the following expression:
\begin{equation}\aligned\nonumber
\frac{\partial^3}{\partial x \partial y^2}\theta(\alpha;z)=&-\frac{1}{2}\alpha^{-\frac{1}{2}}y^{-\frac{3}{2}}\sum_{n=1}^\infty n e^{-\alpha \pi y n^2}\vartheta_{Y}(\frac{y}{\alpha};nx)
-2\pi {\alpha}^{\frac{1}{2}}{y}^{-\frac{1}{2}}\sum_{n=1}^\infty n^3 e^{-\alpha \pi y n^2}\vartheta_{Y}(\frac{y}{\alpha};nx)\\
&+2\pi^2{\alpha}^{\frac{3}{2}}{y}^{\frac{1}{2}}\sum_{n=1}^\infty n^5 e^{-\alpha \pi y n^2}\vartheta_{Y}(\frac{y}{\alpha};nx)
+2{\alpha}^{-\frac{3}{2}}y^{-\frac{1}{2}}\sum_{n=1}^\infty n e^{-\alpha \pi y n^2}\vartheta_{XY}(\frac{y}{\alpha};nx)\\
&-4\pi{\alpha}^{-\frac{1}{2}}y^{\frac{1}{2}}\sum_{n=1}^\infty n^3 e^{-\alpha \pi y n^2}\vartheta_{XY}(\frac{y}{\alpha};nx)
+2{\alpha}^{-\frac{5}{2}}y^{\frac{1}{2}}\sum_{n=1}^\infty n e^{-\alpha \pi y n^2}\vartheta_{XXY}(\frac{y}{\alpha};nx).
\endaligned\end{equation}
\end{lemma}

By Fourier transform, we have $\theta(\frac{1}{\alpha};z)=\alpha \theta(\alpha;z)$ (\cite{Luo2022,Mon1988}). Therefore, the proof of Theorem \ref{Th1} can be reduced to the case when $\alpha\geq1.$ Then, for $\alpha\geq1,$ we provide a lower bound for $\frac{\partial^2}{\partial x\partial y}\theta(\alpha;z)$.

\begin{lemma} \label{3.1lem3} For $\alpha\geq1$ and $z\in \{(x,y)|0<x<\frac{1}{2},y\geq\frac{3}{5}\}$, it holds that
\begin{equation}\aligned\nonumber
\frac{\partial^2}{\partial x \partial y}\theta(\alpha;z)\geq \frac{\pi}{10} \alpha^{\frac{1}{2}}y^{\frac{1}{2}}\big( -\vartheta_{Y}(\frac{y}{\alpha};x) \big)e^{-\pi\alpha y}>0.
\endaligned\end{equation}
\end{lemma}
\begin{proof}
By Lemma \ref{3.1lem2}, $\theta_{xy}(\alpha;z)$ can be rewritten as follows:
\begin{equation}\aligned\label{3.1lem3eq1}
\frac{\partial^2}{\partial x\partial y}\theta(\alpha;z)=&2\pi {\alpha}^{\frac{1}{2}}y^{\frac{1}{2}}e^{-\pi\alpha y}\big(-\vartheta_{Y}(\frac{y}{\alpha};x)  \big)
\big(1-\frac{1}{2\pi\alpha y}-\frac{1}{\pi {\alpha}^2}\frac{\vartheta_{XY}(\frac{y}{\alpha};x)}{\vartheta_{Y}(\frac{y}{\alpha};x)}\\
&+\sum_{n=2}^{\infty}n^3 e^{-\pi\alpha y(n^2-1)}\frac{\vartheta_{Y}(\frac{y}{\alpha};nx)}{\vartheta_{Y}(\frac{y}{\alpha};x)}-\frac{1}{2\pi\alpha y}\sum_{n=2}^{\infty}n e^{-\pi\alpha y(n^2-1)}\frac{\vartheta_{Y}(\frac{y}{\alpha};nx)}{\vartheta_{Y}(\frac{y}{\alpha};x)}\\
&-\frac{1}{\pi{\alpha}^2}\sum_{n=2}^{\infty}n e^{-\pi\alpha y(n^2-1)}\frac{\vartheta_{XY}(\frac{y}{\alpha};nx)}{\vartheta_{Y}(\frac{y}{\alpha};x)}
\big).
\endaligned\end{equation}
By \eqref{lem2.3eq2}, we have $\big(-\vartheta_{Y}(\frac{y}{\alpha};x)>0$ for $0<x<\frac{1}{2}.$ Next, we will divide the proof into two cases: $\frac{y}{\alpha}\geq\frac{1}{2}$ and $\frac{y}{\alpha}\leq\frac{1}{2}.$\\
Case A: $\frac{y}{\alpha}\geq\frac{1}{2}.$ By \eqref{3.1lem3eq1} and Lemma \ref{2lem1}, we have
\begin{equation}\aligned\label{3.1lem3eq2}
\frac{\partial^2}{\partial x \partial y}\theta(\alpha;z)\geq&2\pi {\alpha}^{\frac{1}{2}}y^{\frac{1}{2}}e^{-\pi\alpha y}\big(-\vartheta_{Y}(\frac{y}{\alpha};x)  \big)
\big(1-\frac{1}{2\pi\alpha y}+\frac{1}{ {\alpha}^2}\frac{1+\hat{\nu}(\frac{y}{\alpha})}{1+\hat{\mu}(\frac{y}{\alpha})}\\
&-\frac{1+\mu(\frac{y}{\alpha})}{1-\mu(\frac{y}{\alpha})}\cdot\nu(\alpha y)
-\frac{1}{2\pi\alpha y} \frac{1+\mu(\frac{y}{\alpha})}{1-\mu(\frac{y}{\alpha})}\cdot \mu(\alpha y)-\frac{1}{{\alpha}^2}\frac{1+\nu(\frac{y}{\alpha})}{1-\mu(\frac{y}{\alpha})}\cdot \mu(\alpha y)\big).
\endaligned\end{equation}
Note that $\mu(x),\nu(x),\frac{1+\mu(x)}{1-\mu(x)}$ and $\frac{1+\nu(x)}{1-\mu(x)}$ are decreasing functions , while $\frac{1+\hat{\nu}(x)}{1+\hat{\mu}(x)}$ is increasing function as $x\geq\frac{1}{2}$. Therefore, for $\frac{y}{\alpha}\geq\frac{1}{2}$ and $\alpha y\geq\frac{3}{5},$ by \eqref{3.1lem3eq2}, one gets
\begin{equation}\aligned\label{3.1lem3eq3}
\frac{\partial^2}{\partial x\partial y}\theta(\alpha;z)\geq&2\pi {\alpha}^{\frac{1}{2}}y^{\frac{1}{2}}e^{-\pi\alpha y}\big(-\vartheta_{Y}(\frac{y}{\alpha};x)  \big)
\Big(1+\frac{1}{ {\alpha}^2}\big(\frac{1+\hat{\nu}(\frac{1}{2})}{1+\hat{\mu}(\frac{1}{2})}-\frac{1+\nu(\frac{1}{2})}{1-\mu(\frac{1}{2})}\cdot \mu(\frac{3}{5})\big)\\
&-\frac{1}{2\pi\alpha y}(1+\frac{1+\mu(\frac{1}{2})}{1-\mu(\frac{1}{2})}\cdot \mu(\frac{3}{5}))-\frac{1+\mu(\frac{1}{2})}{1-\mu(\frac{1}{2})}\cdot\nu(\frac{3}{5})\Big).
\endaligned\end{equation}
In \eqref{3.1lem3eq3}, by computation, one has $\frac{1+\hat{\nu}(\frac{1}{2})}{1+\hat{\mu}(\frac{1}{2})}-\frac{1+\nu(\frac{1}{2})}{1-\mu(\frac{1}{2})}\cdot \mu(\frac{3}{5})=0.8729\cdots$ and
\begin{equation}\aligned\label{3.1lem3eq3add}
\frac{1+\mu(\frac{1}{2})}{1-\mu(\frac{1}{2})}\cdot \mu(\frac{3}{5})=0.0150\cdots,\;\;\frac{1+\mu(\frac{1}{2})}{1-\mu(\frac{1}{2})}\cdot\nu(\frac{3}{5})=0.0602\cdots.
\endaligned\end{equation}
Given that \eqref{3.1lem3eq3} and \eqref{3.1lem3eq3add}, one has
\begin{equation}\aligned\label{3.1lem3eq4}
\frac{\partial^2}{\partial x\partial y}\theta(\alpha;z)
\geq& 2\pi {\alpha}^{\frac{1}{2}}y^{\frac{1}{2}}e^{-\pi\alpha y}\big(-\vartheta_{Y}(\frac{y}{\alpha};x)  \big)
\Big(1-\frac{1}{2\pi\alpha y}\big(1+0.0150\big)
-0.0602\Big)\\
\geq&\frac{4\pi}{3} {\alpha}^{\frac{1}{2}}y^{\frac{1}{2}}e^{-\pi\alpha y}\big(-\vartheta_{Y}(\frac{y}{\alpha};x)  \big)>0.
\endaligned\end{equation}
Case B: $\frac{y}{\alpha}\leq \frac{1}{2}.$ Since $y\geq \frac{3}{5}$, then $\alpha\geq 2y\geq\frac{6}{5}$, and
\begin{equation}\aligned\label{3.1lem3eq5}
\frac{1}{\pi}\sum_{n=2}^{\infty}n^4 e^{-\pi \alpha (y(n^2-1)-\frac{1}{4y})}
+(\frac{2}{\pi^2 \alpha y}+\frac{1}{4\pi y^2})\sum_{n=2}^{\infty}n^2 e^{-\pi \alpha (y(n^2-1)-\frac{1}{4y})}
\leq0.039.
\endaligned\end{equation}
By \eqref{3.1lem3eq1} and Lemma \ref{2lem1}, we have
\begin{equation}\aligned\label{3.1lem3eq4}
\frac{\partial^2}{\partial x \partial y}\theta(\alpha;z)\geq &2\pi {\alpha}^{\frac{1}{2}}y^{\frac{1}{2}}e^{-\pi\alpha y}\big(-\vartheta_{Y}(\frac{y}{\alpha};x)  \big)
\big(1-\frac{1}{2\pi\alpha y}-\frac{1}{4y^2}-\frac{1}{\pi}\sum_{n=2}^{\infty}n^4 e^{-\pi \alpha (y(n^2-1)-\frac{1}{4y})}\\
&-(\frac{2}{\pi^2 \alpha y}+\frac{1}{4\pi y^2})\sum_{n=2}^{\infty}n^2 e^{-\pi \alpha (y(n^2-1)-\frac{1}{4y})}
\big).\\
\endaligned\end{equation}
For $y\geq \frac{3}{5}$ and $\alpha\geq\frac{6}{5}$, using \eqref{3.1lem3eq5} and \eqref{3.1lem3eq4}, one has
\begin{equation}\aligned\label{3.1lem3eq6}
\frac{\partial^2}{\partial x \partial y}\theta(\alpha;z)\geq &2\pi {\alpha}^{\frac{1}{2}}y^{\frac{1}{2}}e^{-\pi\alpha y}\big(-\vartheta_{Y}(\frac{y}{\alpha};x)  \big)
\big(1-\frac{1}{2\pi\alpha y}-\frac{1}{4y^2}-0.039\big)\\
\geq &\frac{\pi}{100} \alpha^{\frac{1}{2}}y^{\frac{1}{2}}\big( -\vartheta_{Y}(\frac{y}{\alpha};x) \big)e^{-\pi\alpha y}>0.
\endaligned\end{equation}
Consequently, \eqref{3.1lem3eq3} and \eqref{3.1lem3eq6} jointly lead to the desired result.
\end{proof}

Next, we derive a lower bound for $-\frac{\partial^3}{\partial x \partial y^2}\theta(\alpha;z).$

\begin{lemma} \label{3.2lem2} Assume that $\alpha\geq1.$ Then for $z\in D_{\mathcal{G}}$, it holds that
\begin{equation}\aligned\nonumber
-\frac{\partial^3}{\partial x \partial y^2}\theta(\alpha;z)\geq  \frac{\pi^2}{50}{\alpha}^{\frac{3}{2}}y^{\frac{1}{2}}\big( -\vartheta_{Y}(\frac{y}{\alpha};x)\big)e^{-\pi\alpha y}>0.
\endaligned\end{equation}
\end{lemma}
\begin{proof}
By Lemma \ref{3.2lem1}, we can rewrite the expression of $\theta_{xyy}(\alpha;z)$ as follows:
\begin{equation}\aligned\label{3.2lem2eq1}
&-\frac{\partial^3}{\partial x \partial y^2}\theta(\alpha;z)=2\pi^2{\alpha}^{\frac{3}{2}}y^{\frac{1}{2}}e^{-\pi\alpha y}\big( -\vartheta_{Y}(\frac{y}{\alpha};x)\big)
\Big(  1-\frac{1}{4\pi^2\alpha^2 y^2}-\frac{1}{\pi\alpha y}+\frac{1}{\pi^2\alpha^3y}\frac{\vartheta_{XY}(\frac{y}{\alpha};x)}{\vartheta_{Y}(\frac{y}{\alpha};x)}\\
&\;\;\;\;\;\;\;\;\;\;\;\;\;\;\;\;\;-\frac{2}{\pi \alpha^2}\frac{\vartheta_{XY}(\frac{y}{\alpha};x)}{\vartheta_{Y}(\frac{y}{\alpha};x)}+\frac{1}{\pi^2\alpha^4}\frac{\vartheta_{XXY}(\frac{y}{\alpha};x)}{\vartheta_{Y}(\frac{y}{\alpha};x)}
+\sum_{n=2}^{\infty} \big(( \frac{n}{\pi^2\alpha^3 y}-\frac{2n^3}{\pi \alpha^2} )\frac{\vartheta_{XY}(\frac{y}{\alpha};nx)}{\vartheta_{Y}(\frac{y}{\alpha};x)}
\\
&\;\;\;\;\;\;\;\;\;\;\;\;\;\;\;\;\;+(n^5-\frac{n}{4\pi^2\alpha^2y^2}-\frac{n^3}{\pi\alpha y})\frac{\vartheta_{Y}(\frac{y}{\alpha};nx)}{\vartheta_{Y}(\frac{y}{\alpha};x)}+ \frac{n}{\pi^2\alpha^4}\frac{\vartheta_{XXY}(\frac{y}{\alpha};nx)}{\vartheta_{Y}(\frac{y}{\alpha};x)}\big)e^{-\pi\alpha y(n^2-1)}\Big).\\
\endaligned\end{equation}
Similar to the proof of Lemma \ref{3.1lem3}, we divide the proof into two cases as follows:\\
Case A: $\frac{y}{\alpha}\geq\frac{1}{2}.$ For simplicity, we denote that
\begin{equation}\aligned\label{3.2lem2sigma}
\epsilon_{1}:=&\frac{1+\mu(\frac{y}{\alpha})}{1-\mu(\frac{y}{\alpha})}\big(  \omega(\alpha y)+\frac{\mu(\alpha y)}{4\pi^2\alpha^2y^2}+\frac{\nu(\alpha y)}{\pi\alpha y}\big)
+\frac{1+\nu(\frac{y}{\alpha})}{1-\mu(\frac{y}{\alpha})}\big( \frac{\mu(\alpha y)}{\pi \alpha^3 y}+\frac{2\nu(\alpha y)}{\alpha^2}\big)
+\frac{\mu(\alpha y)}{\alpha^4}\frac{1+\omega(\frac{y}{\alpha})}{1-\mu(\frac{y}{\alpha})}.\\
\endaligned\end{equation}
By \eqref{3.2lem2eq1} and Lemmas \ref{2lem1} and \ref{2lem2}, we have
\begin{equation}\aligned\label{3.2lem2eq2}
-\frac{\partial^3}{\partial x \partial y^2}\theta(\alpha;z)\geq&2\pi^2{\alpha}^{\frac{3}{2}}y^{\frac{1}{2}}e^{-\pi\alpha y}\big( -\vartheta_{Y}(\frac{y}{\alpha};x)\big)
\Big(1+(2\cdot\frac{1+\hat{\nu}(\frac{y}{\alpha})}{1+\hat{\mu}(\frac{y}{\alpha})}
-\frac{1}{4\pi^2y^2}-\frac{1}{\pi\alpha y}\frac{1+\nu(\frac{y}{\alpha})}{1+\mu(\frac{y}{\alpha})})\frac{1}{\alpha^2}\\
&+\frac{1}{\alpha^4}\frac{1+\hat{\omega}(\frac{y}{\alpha})}{1+\hat{\mu}(\frac{y}{\alpha})}-\frac{1}{\pi\alpha y}-\epsilon_{1}\Big).
\endaligned\end{equation}
Notice that $\mu(x),\nu(x),\omega(x),\frac{1+\nu(x)}{1+\mu(x)},\frac{1+\mu(x)}{1-\mu(x)},\frac{1+\nu(x)}{1-\mu(x)}$ and $\frac{1+\omega(x)}{1-\mu(x)}$ are decreasing functions, while $\frac{1+\hat{\nu}(x)}{1+\hat{\mu}(x)}$ and $\frac{1+\hat{\omega}(x)}{1+\hat{\mu}(x)}$ are increasing function as $x\geq\frac{1}{2}$. Hence $\frac{1+\nu(x)}{1+\mu(x)}\leq \frac{1+\nu(\frac{1}{2})}{1+\mu(\frac{1}{2})}=1.1042\cdots,$ and
\begin{equation}\aligned\label{3.2lem2eq3}
\frac{1+\hat{\nu}(x)}{1+\hat{\mu}(x)}\geq \frac{1+\hat{\nu}(\frac{1}{2})}{1+\hat{\mu}(\frac{1}{2})}=0.8884\cdots,
\frac{1+\hat{\omega}(x)}{1+\hat{\mu}(x)}\geq \frac{1+\hat{\omega}(\frac{1}{2})}{1+\hat{\mu}(\frac{1}{2})}=0.4435\cdots.
\endaligned\end{equation}
As $\frac{y}{\alpha}\geq\frac{1}{2}$, $\alpha\geq1$ and $y\geq\frac{\sqrt{3}}{2}$, by \eqref{3.2lem2eq3}, we have
\begin{equation}\aligned\label{3.2lem2eq4}
2\cdot\frac{1+\hat{\nu}(\frac{y}{\alpha})}{1+\hat{\mu}(\frac{y}{\alpha})}
-\frac{1}{4\pi^2y^2}-\frac{1}{\pi\alpha y}\frac{1+\nu(\frac{y}{\alpha})}{1+\mu(\frac{y}{\alpha})}
\geq 2\cdot\frac{1+\hat{\nu}(\frac{1}{2})}{1+\hat{\mu}(\frac{1}{2})}
-\frac{1}{4\pi^2y^2}-\frac{1}{\pi\alpha y}\frac{1+\nu(\frac{1}{2})}{1+\mu(\frac{1}{2})}>0,
\endaligned\end{equation}
and
\begin{equation}\aligned\label{3.2lem2eq5}
\epsilon_{1}\leq &\frac{1+\mu(\frac{1}{2})}{1-\mu(\frac{1}{2})}\big(  \omega(\frac{\sqrt{3}}{2})+\frac{\mu(\frac{\sqrt{3}}{2})}{4\pi^2\alpha^2y^2}+\frac{\nu(\frac{\sqrt{3}}{2})}{\pi\alpha y}\big)
+\frac{1+\nu(\frac{1}{2})}{1-\mu(\frac{1}{2})}\big( \frac{\mu(\frac{\sqrt{3}}{2})}{\pi \alpha^3 y}+\frac{2\nu(\frac{\sqrt{3}}{2})}{\alpha^2}\big)
+\frac{\mu(\frac{\sqrt{3}}{2})}{\alpha^4}\frac{1+\omega(\frac{1}{2})}{1-\mu(\frac{1}{2})}\\
\leq &\frac{2}{50}.
\endaligned\end{equation}
By using \eqref{3.2lem2eq2}-\eqref{3.2lem2eq5}, we obtain that
\begin{equation}\aligned\nonumber
-\frac{\partial^3}{\partial x \partial y^2}\theta(\alpha;z)\geq&2\pi^2{\alpha}^{\frac{3}{2}}y^{\frac{1}{2}}e^{-\pi\alpha y}\big( -\vartheta_{Y}(\frac{y}{\alpha};x)\big)
(1-\frac{1}{\pi\alpha y}-\epsilon_{1})\\
\geq&\pi^2{\alpha}^{\frac{3}{2}}y^{\frac{1}{2}}e^{-\pi\alpha y}\big( -\vartheta_{Y}(\frac{y}{\alpha};x)\big)>0.
\endaligned\end{equation}
Case B: $\frac{y}{\alpha}\leq\frac{1}{2}.$ In this case, $y\geq\frac{\sqrt{3}}{2},\alpha\geq2y\geq\sqrt{3}$. We denote that
\begin{equation}\aligned\label{3.2lem2sigma}
\epsilon_{2}:=\epsilon_{2}(\alpha;y)=&\frac{1}{\pi}\sum_{n=2}^{\infty}n^6e^{-\pi\alpha\big(y(n^2-1)-\frac{1}{4y} \big)}+\big(\frac{4}{\pi^2\alpha y}+\frac{1}{2\pi y^2} \big)\sum_{n=2}^{\infty}n^4e^{-\pi\alpha\big(y(n^2-1)-\frac{1}{4y}\big)}\\
&+\big(  \frac{11}{2\pi^3\alpha^2y^2}+\frac{1}{4\pi^2\alpha y^3}+\frac{1}{4\pi^2 y^4}\big)\sum_{n=2}^{\infty}n^2e^{-\pi\alpha\big(y(n^2-1)-\frac{1}{4y} \big)}.
\endaligned\end{equation}
By \eqref{3.2lem2eq1} and Lemmas \ref{2lem1} and \ref{2lem2}, we have
\begin{equation}\aligned\label{3.2lem2eq8}
-\frac{\partial^3}{\partial x \partial y^2}\theta(\alpha;z)\geq&2\pi^2{\alpha}^{\frac{3}{2}}y^{\frac{1}{2}}e^{-\pi\alpha y}\big( -\vartheta_{Y}(\frac{y}{\alpha};x)\big)
\big(  1+\frac{7}{2\pi^2\alpha^2 y^2}-\frac{1}{\pi\alpha y}\\
&-\frac{1}{\pi^2\alpha^3y}\frac{\frac{3}{2}(\frac{y}{\alpha})^2+4\pi^2 e^{-\frac{\pi\alpha}{y}}}{(\frac{y}{\alpha})^3-4\pi(\frac{y}{\alpha})^2e^{-\frac{\pi\alpha}{y}}}
-\frac{1}{2y^2}-\frac{1}{4\pi y^4}-\epsilon_{2}\big).\\
\endaligned\end{equation}
Since $\frac{\alpha}{y}\geq2$, it follows that
\begin{equation}\aligned\label{3.2lem2eq9}
\frac{\frac{3}{2}(\frac{y}{\alpha})^2+4\pi^2 e^{-\frac{\pi\alpha}{y}}}{(\frac{y}{\alpha})^3-4\pi(\frac{y}{\alpha})^2e^{-\frac{\pi\alpha}{y}}}
=\frac{\frac{3}{2}\frac{\alpha}{y}+4\pi(\frac{\alpha}{y})^3e^{-\frac{\pi\alpha}{y}}}{1-4\pi\frac{\alpha}{y}e^{-\pi\frac{\alpha}{y}}}
=\frac{\frac{3}{2}+4\pi(\frac{\alpha}{y})^2e^{-\frac{\pi\alpha}{y}}}{1-4\pi\frac{\alpha}{y}e^{-\pi\frac{\alpha}{y}}}\frac{\alpha}{y}
\leq\frac{\frac{3}{2}+16\pi e^{-2\pi}}{1-8\pi e^{-2\pi}}\frac{\alpha}{y}.
\endaligned\end{equation}
In view of \eqref{3.2lem2sigma}$, \epsilon_{2}$ exhibits exponential decay and is monotonically decreasing with respect to $\alpha$ and $y$, hence it can be bounded. Specifically, $\epsilon_{2}\leq 10^{-4}.$ For $\alpha\geq\sqrt{3}$ and $y\geq \frac{\sqrt{3}}{2}$, with \eqref{3.2lem2eq8}-\eqref{3.2lem2eq9}, we have
\begin{equation}\aligned\nonumber
-\frac{\partial^3}{\partial x \partial y^2}\theta(\alpha;z)\geq &2\pi^2{\alpha}^{\frac{3}{2}}y^{\frac{1}{2}}e^{-\pi\alpha y}\big( -\vartheta_{Y}(\frac{y}{\alpha};x)\big)
\Big(1+(\frac{7}{2}-\frac{\frac{3}{2}+16\pi e^{-2\pi}}{1-8\pi e^{-2\pi}})\frac{1}{\pi^2\alpha^2y^2} \\
&-\frac{1}{\pi\alpha y}-\frac{1}{2y^2}-\frac{1}{4\pi y^4}-\epsilon_{2}\Big)\\
\geq& \frac{1}{50}\pi^2{\alpha}^{\frac{3}{2}}y^{\frac{1}{2}}e^{-\pi\alpha y}\big( -\vartheta_{Y}(\frac{y}{\alpha};x)\big)>0.
\endaligned\end{equation}
\end{proof}

The proof of Theorem \ref{Th1} concerning theta function is done by Lemmas \ref{3.1lem3} and \ref{3.2lem2}. The proof of Theorem \ref{Th1} regarding zeta function can be deduced by the result of theta function. In fact, as established by Montgomery \cite{Mon1988} and Luo-Wei \cite{Luo_arxiv_2024-07-08}, there exists a significant relationship between Epstein zeta function and theta function:
\begin{equation}\aligned\label{Th1zeta1}
\zeta(s;z)=\frac{\pi^s}{\Gamma(s)}\int_{0}^{\infty}\big( \theta(\alpha;z)-1 \big)\alpha ^{s-1} d\alpha.
\endaligned\end{equation}
Then taking the mixed partial derivative of both sides of \eqref{Th1zeta1}, we get
\begin{equation}\aligned\label{Th1zeta2}
 \zeta_{xy}(s;z)=\frac{\pi^s}{\Gamma(s)}\int_{0}^{\infty} \theta_{xy}(\alpha;z) \alpha ^{s-1} d\alpha.
\endaligned\end{equation}
\eqref{Th1zeta2} indicates that $\zeta_{xy}(s;z)$ and $\theta_{xy}(\alpha;z)$ are of the same sign. Similarly, $\zeta_{xyy}(s;z)$ maintain the same sign as $\theta_{xyy}(\alpha;z).$ These complete the proof.

\bigskip
{\bf Acknowledgements.}
The research of S. Luo is partially supported by the National Natural Science Foundation of China (NSFC) under Grant Nos. 12261045 and 12001253, and by the Jiangxi Jieqing Fund under Grant No. 20242BAB23001.

\bigskip


\end{document}